\documentclass[reqno,a4paper]{amsart}
\usepackage[latin1]{inputenc}
\usepackage[T1]{fontenc}
\usepackage[english]{babel}
\usepackage[allcolors=blue,colorlinks=true, pdfstartview=FitH, linkcolor=blue, pdftoolbar=true, bookmarks=true,bookmarksnumbered,plainpages,backref]{hyperref}
\usepackage[usenames,dvipsnames,svgnames,table]{xcolor}
\usepackage{graphicx}
\usepackage{caption}
\usepackage{subcaption}
\usepackage[normalem]{ulem}
\usepackage{bbold}
\usepackage{indentfirst}
\usepackage{tablefootnote}
\usepackage[lmargin=2.5cm,tmargin=2cm,bmargin=2cm,rmargin=2.5cm]{geometry}
\usepackage[usenames,dvipsnames,svgnames,table]{xcolor}
\usepackage{amsmath,amssymb,amsthm,amsfonts,esint}
\DeclareMathAlphabet{\mathcal}{OMS}{cmsy}{m}{n}
\usepackage{lipsum}
\usepackage{footnote}
\usepackage{hyperref}
\usepackage{mathtools}
\usepackage[normalem]{ulem}

\usepackage{newtxtext}

\usepackage{enumitem}

\newtheoremstyle{theorem}
{6pt +1\p@ -2.0\p@}
{6pt +1\p@ -2.0\p@}
{\it}			      
{}				  
{\bfseries}   
{.}               
{.4em}       
{}               
\theoremstyle{theorem}
\newtheorem{theorem}{Theorem}
\newtheorem{definition}{Definition}
\newtheorem{proposition}{Proposition}
\newtheorem{corollary}{Corollary}
\newtheorem{lemma}{Lemma}
\newtheorem{remark}{Remark}

\newcommand{\R}{\mathbb R}
\newcommand{\Rn}{{\R^n}}

\newcommand{\N}{\mathbb N}
\newcommand{\Z}{\mathbb Z}
\newcommand{\C}{\mathbb C}

\newcommand{\fM}{\mathfrak{M}}
\newcommand{\fC}{\mathfrak{C}}

\newcommand{\Ga}{\alpha}

\newcommand{\supp}{\text{supp}\,}

\begin{document}
	
\title{Necessary cancellation conditions for the boundedness of operators on local Hardy spaces}

	\author{Galia Dafni}
\address{Department of Mathematics and Statistics, Concordia University, Montreal, QC, H3G 1M8, Canada}
\email{galia.dafni@concordia.ca}

\author{Chun Ho Lau}
\address{Department of Mathematics and Statistics, Concordia University, Montreal, QC, H3G 1M8, Canada}
\email{chunho.lau@concordia.ca}

\author {Tiago Picon}
\address{Departamento de Computa\c{c}\~ao e Matem\'atica, Universidade S\~ao Paulo, Ribeir\~ao Preto, SP, 14040-901, Brasil}
\email{picon@ffclrp.usp.br}

\author {Claudio Vasconcelos}
\address{Departamento de Matem\'atica, Universidade Federal de S\~ao Carlos, S\~ao Carlos, SP, 13565-905, Brasil}
\email{claudio.vasconcelos@estudante.ufscar.br}

\thanks{The first and second author were partially supported by the Natural Sciences and Engineering Research Council (NSERC) of Canada and the Centre de Recherches Math\'{e}matiques (CRM). The third author was supported by Conselho Nacional de Desenvolvimento Cient\'ifico e Tecnol\'ogico (CNPq - grant 311430/2018-0) and Funda\c{c}\~ao de Amparo \`a  Pesquisa do Estado de S\~ao Paulo (FAPESP - grant 18/15484-7). The fourth author was supported by Coordena\c{c}\~ao de Aperfei\c{c}oamento de Pessoal de N\'ivel Superior (CAPES), the Fonds de Recherche du Qu\'ebec - Nature and Technologies (FRQNT) and MITACS Globalink}



\begin{abstract}
		In this work we present necessary cancellation conditions for the continuity of linear operators in $h^p(\Rn)$, $0<p\leq 1$, that map atoms into pseudo-molecules. Our necessary condition, expressed in terms of the $T^{\ast}$ condition, is the same as the  one recently proved  sufficient in \cite{DLPV}, thus providing a necessary and sufficient cancellation condition for the boundedness of inhomogeneous Calder\'on--Zygmund type operators.
\end{abstract}

\maketitle

\section{Introduction}

When studying the boundedness of an operator $T$ on the Hardy space $H^p(\Rn)$, $0 < p \leq 1$, one needs to guarantee that the image of an $H^p$ distribution under $T$ satisfies the cancellation conditions required of elements of $H^p(\Rn)$, namely vanishing moments of all orders $\alpha$ with  $|\alpha| \leq n(1/p-1)$. 
When $T$ is a classical Calder\'on-Zygmund operator, necessary and sufficient cancellation condition for the boundedness were presented in \cite[Proposition 4 p. 23]{MeyerCoifman1997} and expressed in terms of $T^{\ast}(x^{\alpha})$. In the case $p = 1$, for example, one must have $\int Tf = 0$ for $f \in H^1(\Rn)$, or, as commonly stated, $T^*(1) = 0$.  
 
What can be said about the cancellation conditions for bounded operators on the local Hardy spaces $h^p(\Rn)$ introduced by Goldberg \cite{Goldberg1979}? Unlike the case of $H^p(\Rn)$, the elements of $h^p(\Rn)$ are not required to satisfy exact (homogeneous) global cancellation conditions in the form of vanishing moments. Nevertheless, there is an underlying {\em local} or {\em nonhomogeneous} cancellation. Goldberg exhibits an atomic decomposition for $h^p(\Rn)$, analogous to that of $H^p(\Rn)$, with the difference that only atoms supported in small balls are required to have vanishing moments.  Conversely, as we will show in Proposition~\ref{lemma:moment-log-control}, for $h^p(\Rn)$ distributions with compact support, it is possible to control the moments in terms of the norm and the size of the support.  In particular, when $n({1}/{p} - 1)$ is an integer, the highest order moments must decay logarithmically with the size of the support.   
Looking again at the special case $p = 1$, this means that a function $g$ in $h^1(\Rn)$ with support in a ball of radius $r < 1$ must satisfy
\begin{equation}
\label{eq:p=1}
\left|\int g\right| \leq C \bigg[\log \left(1+\frac{1}{r}\right)\bigg]^{-1} \|g\|_{h^1}.
\end{equation}

Since a bounded operator $T$ on $h^{p}(\Rn)$ does not, in general, preserve compact support, a natural question arises concerning the cancelation conditions for $Ta$ where $a$ is an atom.  When  proving boundedness on Hardy spaces $H^{p}(\Rn)$, the usual method is to show that $T$ maps atoms into molecules (see \cite{Taibleson-Weiss, RodolfoTorresBook}), as in the context of Calder\'{o}n-Zygmund operators (see \cite[Proposition 4]{MeyerCoifman1997}). In this particular setting, the notion of molecules is motivated by  the behavior of the kernel associated to the operator. 
For the homogeneous Triebel--Lizorkin spaces $\dot{F}^{\alpha,q}_p(\Rn)$, $1<p<\infty$,  \cite[Theorem 1.16]{Frazier-Jawerth-Han-Weiss} asserts that if a continuous operator maps smooth atoms into a special type of smooth molecules, then the kernel of this operator satisfies Calder\'on--Zygmund estimates. The case $\dot{F}^{0,2}_1(\Rn) = H^1(\Rn)$ is addressed in \cite[Remark (ii) p.\,180]{Frazier-Jawerth-Han-Weiss}, which points out that the cancellation condition $\int Ta(x)dx = 0$, corresponding to $T^*(1) = 0$, guarantees that such an operator maps smooth atoms into  smooth molecules.

In recent work \cite{DLPV}, we presented atomic and molecular decompositions for $h^p(\Rn)$ in which vanishing moments were replaced by {\em approximate} cancellation conditions. In particular, we used these decompositions to prove the boundedness on $h^p(\Rn)$, for all $0 < p \leq 1$, of operators known as inhomogeneous Calder\'on--Zygmund operators (see \cite{DLPV, DingHanZhu2020, RodolfoTorresBook}), and their strongly singular versions.
As part of the sufficient conditions for such an operator to be bounded  from $h^p(\Rn)$ to itself, we imposed local Campanato-type cancellation conditions  $T$.

Our main theorem in the present paper shows that the  local Campanato-type cancellation conditions are also necessary for the boundedness of operators in these classes. More generally, this can be stated for operators that map atoms into what we call pseudo-molecules (see Definition~\ref{pseudo-molecule}).

\begin{theorem} \label{theorem:converse}
	Let $0<p\leq 1$ and $T$ be a linear and bounded operator on $h^p(\R^n)$ that maps each $(p,2)$ atom in $h^p(\Rn)$ into a pseudo-molecule centered in the same ball as the support of the atom.  Then the following cancellation conditions must hold:
	\begin{quote}
For any ball $B=B(x_0,r) \subset \R^n$ with $r<1$ and $\alpha \in \Z_{+}^{n}$ such that $|\alpha|\leq N_p:= \lfloor \gamma_p \rfloor$, $\gamma_p := n(\frac{1}{p} - 1)$,
	\begin{equation} \label{inhomogeneous-condition}
		f= T^*[(\cdot-x_0)^{\alpha}] \quad \mbox{ satisfies } \quad \left(\fint_{B}|f(y)-P^{N_p}_{B}(f)(y)|^2dy\right)^{1/2}\leq C \, \Psi_{p,\alpha}(r),
	\end{equation}
	where $P_{B}^{N_p}(f)$ is the  polynomial of degree less then or equal to $N_p$ that has the same moments as $f$ over $B$ up to order $N_p$, and 
	$$
	\Psi_{p,\alpha}(t) :=\left\{ \begin{array}{ll} 
		t^{\gamma_p} &\quad \text{if } |\alpha|< \gamma_p, \\
		t^{\gamma_p} \bigg[ \log\bigg(1+\frac{1}{t}\bigg)\bigg]^{-1/p} &\quad \text{if } |\alpha| = \gamma_p = N_p.
	\end{array} \right.
	$$
\end{quote}
\end{theorem}

Since a special case of pseudo-molecules are pre-molecules (see Definition~\ref{definition:pre-molecule}, Lemma~\ref{example:premolecule}), we obtain the following $T^{*}$ characterization result, in the spirit of  \cite[Proposition 4]{MeyerCoifman1997}.

\begin{corollary} 
\label{corollary:converse-ICZOp}
Let $0<p\leq 1$ and consider a linear and continuous operator $T:\mathcal{S}'(\Rn) \rightarrow \mathcal{S}'(\Rn)$. Suppose $T$ maps each $(p,2)$ atom in $h^p(\Rn)$ into a pre-molecule centered in the same ball as the support of the atom.  Then the cancellation conditions \eqref{inhomogeneous-condition} hold if and only if $T$ is bounded on $h^{p}(\Rn)$.
\end{corollary}

As a consequence of \cite[Theorem 5.3 and 5.8]{DLPV} and the previous corollary, we can then state necessary and sufficient conditions for the boundedness of (strongly singular) inhomogeneous Calder\'on--Zygmund operators on $h^p(\Rn)$ for all $p \in (0,1]$.

\begin{theorem} \label{theorem:thm2}
	Let $0<p\leq 1$ and $T$ a (strongly singular) inhomogeneous Calder\'on--Zygmund operator. Then $T$ is bounded on $h^p(\R^n)$ if and only if \eqref{inhomogeneous-condition} holds for every ball $B$ in $\R^n$. 
\end{theorem}

In the range $\frac{n}{n+1} <p<1$, sufficient conditions for the continuity of inhomogeneous Calder\'on--Zygmund operators in $h^p(\Rn)$ were studied in \cite[Theorem 1.1]{DingHanZhu2020}. The condition $T^{\ast}(1) \in \dot{\Lambda}_{n(1/p-1)}$ was used to show that $T$ maps atoms into molecules, followed by an application of the molecular decomposition for $h^p(\Rn)$  given by Komori in \cite{Komori2001}, valid for $\frac{n}{n+1} <p<1$. The molecules in this case resemble the homogeneous case, except that the vanishing integral condition is replaced by a uniform estimate of its size, namely $|\int M| \leq C$ (see \cite[Definition 4.4]{Komori2001}). Conversely,  in \cite{DingHanZhu2020} the authors used the bound
\begin{equation}
 \label{eq:condition-Ding-Han-Zhu}
\bigg|\int f\bigg| \lesssim \| f \|_{h^p}, \;\; \;\; f\in L^2(\Rn)\cap h^p(\Rn),
\end{equation}
to show, via duality, that if $T$ is bounded on $h^p(\Rn)$ for $\frac{n}{n+1}<p<1$, then $T^{*}(1) \in \Lambda_{n(1/p-1)}$, the inhomogeneous Lipschitz space that is the dual of $h^p(\Rn)$ in this range.

Looking at $p =1$, one sees that the Komori condition is not a sufficiently strong moment condition, and \eqref{eq:condition-Ding-Han-Zhu} follows from the embedding of $h^1(\Rn)$ in $L^1(\Rn)$.  As pointed out above, the stronger vanishing condition \eqref{eq:p=1} holds for functions with compact support.  In \cite{DLPV}, a molecular theory for $h^p(\Rn)$, for all $0<p\leq 1$, was presented, covering the case studied by Komori and requiring the logarithmic decay of the highest-order moments when $p=\frac{n}{n+k}$, for every $k\in\Z_{+}$. These conditions are shown to be necessary in Proposition~\ref{proposition:pseudo-molecule-moment-estimate}.

Very recently, molecules and cancellation conditions for the boundedness of inhomogeneous Calder\'on--Zygmund operators were also studied by Bui and Ly in \cite{Bui-Ly-preprint}.  The moments of the molecules in \cite[Definition 2.2]{Bui-Ly-preprint} are required to decay like a power of the radius of the associated ball, a stronger condition previously introduced in  \cite[Appendix B]{GaliaThesis}. Furthermore, the sufficient conditions for the boundedness in \cite[Theorem 1.2]{Bui-Ly-preprint} differ significantly from the necessary conditions (again based on a duality argument), and do not apply in the cases {$p=\frac{n}{n+s}$, $s$ an integer,} which are precisely the cases covered by our result.

The organization of the paper is as follows. In Section \ref{section:preliminaries}, we provide basic definitions and facts about Hardy 
spaces that will be used in this work. Section \ref{section:proof-theorem} is devoted to showing the necessity of cancellation conditions.  It contains the definitions of pseudo-molecules, pre-molecules and inhomogeneous Calder\'on--Zygmund operators, and culminates in the proofs of Theorems~\ref{theorem:converse} and \ref{theorem:thm2}.

\section{Preliminaries} \label{section:preliminaries}

Throughout this paper, we denote by $\Rn$ the $n$-dimensional Euclidean space and $B(x_0,r)$ a ball in $\Rn$ centered at $x_0 \in \Rn$ with radius $r>0$. When omitting the center and the radius, denoting the ball only by $B$, we mean a generic ball in $\Rn$. By $\N$ and $\Z_{+}$ we denote the positive and  nonnegative integers, respectively. Given a locally integrable function $f$, we write
$$
f_B := \fint_{B} f(x)dx := \frac{1}{|B|} \int_{B} f(x)dx 
$$
for its mean over the ball $B$, in which $|B|$ denotes the Lebesgue measure of $B\subset \Rn$. The notation $f \lesssim g$ means that there exists a geometric constant $C>0$ such that $f \leq C \, g$. For $0<p\leq 1$, write $\gamma_p := n(\frac{1}{p} - 1)$ and $N_p := \lfloor \gamma_p \rfloor$, where $\lfloor \, \cdot \, \rfloor$ is the floor function.

The spaces $h^p(\Rn)$ for $p>0$ were introduced by Goldberg \cite{Goldberg1979}. For a given $\varphi \in \mathcal{S}(\Rn)$ such that $\int{\varphi(x)dx} \neq 0$ and $t>0$, let $\varphi_t(x) = t^{-n}\varphi(t^{-1}x)$. We say that $f \in \mathcal{S}'(\Rn)$ lies in $h^p(\Rn)$ if
$$
\|f\|_{h^p} := \|m_{\varphi}\|_{L^p} < \infty, \;\; \text{where} \;\;	m_{\varphi}f(x) := \sup_{0<t<1} \left| \langle f,\, \varphi_t(x-\cdot)  \rangle \right|.
$$
The functional $\| \cdot \|_{h^p}$ defines a norm for $p\geq 1$ and a quasi-norm otherwise. We refer to it always as a norm for simplicity. Even though we start with a fixed $\varphi$, the local Hardy spaces remains the same no matter which $\varphi$ we choose. It is well known that $h^p(\R^n) = L^p(\Rn)$ with equivalent norms if $p > 1$, and $\mathcal{S}(\Rn) \subset h^1(\Rn) \subsetneq L^1(\Rn)$ continuously. 

As a consequence of the relationship between $h^p(\Rn)$ and $H^p(\Rn)$, see \cite[Lemma 4]{Goldberg1979}, Goldberg showed that elements of local Hardy spaces can be decomposed into atoms in which no moment condition is required when the atom is supported in large balls. 

\begin{definition} 
\label{goldberg-atom}
	Let $0<p\le 1 \le s \le \infty$ with $p\neq s$. A measurable function $a$ is called a $(p,s)$ atom (for $h^p(\Rn)$) if there exists a ball $B=B(x_0,r)\subset \Rn$ such that
	\begin{equation*}
		\textnormal{(i)} \;\; \supp(a) \subset B; \quad \textnormal{(ii)} \;\; \| a \|_{L^s} \le r^{n\left(\frac{1}{s}-\frac{1}{p}\right)}; \quad \textnormal{(iii)} \;\; \textnormal{If} \ r<1, \int{a(x)x^{\alpha}}dx=0 \ \textnormal{for all} \ |\alpha|\le N_p \ . 
	\end{equation*}
\end{definition}

A function $a$ satisfying the vanishing moment condition (iii) regardless of the size of its support, in addition to  (i) and (ii), will be called $(p,s)$ atom for $H^p(\Rn)$.

In \cite[Lemma 5]{Goldberg1979}, it was shown for $s = \infty$ that if $f \in h^p(\Rn)$, then there exists a sequence $\{a_j\}_{j \in \N}$ of $(p,s)$ atoms in $h^p(\Rn)$ and a sequence $\{\lambda_j\}_{j \in \N}$ of complex scalars in $\ell^p(\C)$ such that 	
\begin{equation} \label{goldberg-decomp}
	f = \sum_{j \in \N} \lambda_j \, a_j \ \ \mbox{in $h^p$, and} \ \ \| f \|_{h^p} \approx \inf  \bigg( \sum_{j \in \N} |\lambda_j|^p \bigg)^{{1}/{p}},  
\end{equation}
where the infimum is taken over all such representations.

We now give the characterization of $h^p(\Rn)$ by the grand maximal function, and also replace the restriction $0 < t < 1$ in the definition of the maximal function by $0 < t < T$ for some $T < \infty$, which results in equivalent norms. Given $0<T<\infty$ and $x\in \Rn$, consider the family
$$
\mathcal{F}^{T,\,x}_{k} = \left\{ \phi \in \, C^\infty(\Rn): \ \supp(\phi) \subset B(x,t), \ 0<t<T \ \text{and} \ \|\partial^{\alpha}\phi\|_{L^{\infty}} \leq t^{-n-|\alpha|} \;\; \text{for all} \;\; |\alpha|\leq k \right\}.
$$
We define the \textit{local grand maximal function} associated to the family $\mathcal{F}^{T,\,x}_{k}$ by
$$
m_{\mathcal{F}_{k}}(f)(x)=\sup_{\phi \, \in \, \mathcal{F}^{T,x}_{k}} \left| \langle f,\,\phi\rangle \right|.
$$ 
\begin{lemma} \label{lemma:local-grand-maximal}
	Let $f\in L^{1}_{loc}(\Rn)$. If $k \in \N$ is such that $\frac{n}{n+k}<p\leq \frac{n}{n+k-1}$ (i.e.\ $k = N_p + 1$) then
	\begin{equation} \label{eq:bound-norm-hp}
	\|m_{\mathcal{F}_{k}}(f) \|_{L^p} \le C_{n,p,T} \, \| f \|_{h^p},	
	\end{equation}
	where $C_{n,1,T} \lesssim 1 + \log_+T$ and $C_{n,p,T} \lesssim \max \{1,\, T^{n(1/p - 1)}\}$ for $p < 1$.
\end{lemma}

\begin{proof}
Since the atomic decomposition \eqref{goldberg-decomp} converges in the sense of distributions, and $m_{\mathcal{F}_{k}}$ is sub-linear, it suffices to prove that
$$
	\| m_{\mathcal{F}_{k}} (a) \|_{L^p} \leq C
	$$
for a $(p, \infty)$ atom $a$ supported on some ball $B=B(x_0,r) \subset \Rn$. 	Indeed, writing $f=\sum_{j \in \N} \lambda_j \, a_j$, this will give
	$$
	\| m_{\mathcal{F}_{k}}(f) \|_{L^p} \leq \bigg( \sum_{j\in\N} |\lambda|^p \, \| m_{\mathcal{F}_{k}}(a) \|_{L^p}^{p}  \bigg)^{1/p} \leq C \bigg( \sum_{j\in\N} |\lambda|^p \bigg)^{1/p}$$
	and we can take the decomposition so that the right-hand-side is bounded by a constant multiple of $\| f \|_{h^p}$.

	So fix $a$ and split
	$$
	\| m_{\mathcal{F}_{k}} (a) \|_{L^p}^{p} = \int_{B(x_0,2r)} [m_{\mathcal{F}_{k}} (a)(x)]^pdx + \int_{\Rn \setminus B(x_0,2r)} [m_{\mathcal{F}_{k}} (a)(x)]^pdx.
	$$
	To deal with the first integral, note that for any $\phi \in \mathcal{F}_{k}^{T,\,x}$ one has
	\begin{align*}
		\left|\int a\phi \right| \leq \| a\|_{L^{\infty}} \, \| \phi \|_{L^{\infty}} \, |B(x_0,r)\cap B(x,t)| \leq C_{n} \, r^{-\frac np}.
	\end{align*}
	Then
	\begin{align*}
		\int_{B(x_0,2r)} [m_{\mathcal{F}_{k}} (a)(x)]^pdx \leq C_{n,p} \, r^{-n} \, |B(x_0,2r)| \simeq C_{n,p}.
	\end{align*}

	For the non-local case, when $x \not\in B(x_0,2r)$, note that $\int a\phi$ vanishes unless  $B(x,t) \cap B(x_0,r)\neq \emptyset$, hence $r \leq \frac{|x-x_0|}{2} < t < T$. Thus, if $r \geq 1$ we have
	$$\left|\int a\phi \right| \leq   \| a \|_{L^1}  \| \phi \|_{L^{\infty}} \leq C_{n} \,  r^{n\left(1-\frac{1}{p}\right)} t^{-n} \leq C_{n} |x-x_0|^{-n},$$
	and therefore 
	$$\int_{\Rn \setminus B(x_0,2r)} [m_{\mathcal{F}_{k}} (a)(x)]^pdx \lesssim \int_{2r<|x-x_0|<2T} |x-x_0|^{-np}dx \\
		\lesssim \int_{2<|x-x_0|<2T} |x-x_0|^{-np}dx < \infty.$$
		Note that the integral on the right is of the order of $\log T$ when $p = 1$ and $T^{n(1 - p)}$ when $p < 1$.
	
	For $0<r<1$, we have the standard $H^p(\Rn)$ argument, using the moment conditions of $a$ up to the order $N_p = k-1$ and the Taylor expansion of $\phi \in \mathcal{F}^{T,\,x}_{k}$ to write
	\begin{align*}
		\left|\int a(y) \phi(y)dy \right| &= \bigg| \int \bigg[ \phi(y)-\sum_{|\alpha| \leq k-1} C_{\alpha} \, \partial^{\alpha} \phi(x-x_0)(y-x_0)^{\alpha}  \bigg] a(y) dy \bigg| \\
		& \leq \sum_{|\alpha| = k} C_{\alpha} \, \| \partial^{\alpha} \phi \|_{L^{\infty}}  \, r^{|\alpha|+n} \, \| a \|_{L^{\infty}} \\
		&\leq C_{n} \, t^{-n-k} \, r^{k+n\left(1-\frac{1}{p}\right)}.
	\end{align*}
	Then
	\begin{align*}
		\int_{\Rn \setminus B(x_0,2r)} [m_{\mathcal{F}_{k}} (a)(x)]^pdx \leq C_{n,p} \, r^{kp+np-n} \int_{|x-x_0|>2r} |x-x_0|^{p(-k-n)}dx < \infty,
	\end{align*}
	since $p>n/(n+k)$.
\end{proof}

\begin{remark}
	\textnormal{It is also possible to show the other direction of \eqref{eq:bound-norm-hp}, since $m_{\varphi}f \leq C \, m_{\mathcal{F}_{k}}f$.}
\end{remark}

\section{The necessity of the cancellation conditions} 
\label{section:proof-theorem}
	
Our first result is a strengthening  of \eqref{eq:condition-Ding-Han-Zhu} for $f\in h^p(\Rn)$ supported in small balls,
where more moments are considered. In particular, a more appropriate logarithmic bound, depending on the ball $B$, is provided when $p=\frac{n}{n+k}$ for some $k \in \Z_{+}$, that is $\gamma_p \in \Z_{+}$.  
	
\begin{proposition} 
\label{lemma:moment-log-control}
	Let $g \in h^p(\R^n)$ be supported in $B(x_0,r)$ for some $x_0 \in \Rn$ and $0<r<1$.  Then for $\alpha \in \Z_+^n$, the moments $\langle g, \, (\cdot-x_0)^{\alpha} \rangle$ are well-defined and satisfy
	\begin{equation} \label{eq:molecule-ball}
		\left| \langle g, \, (\cdot-x_0)^{\alpha} \rangle \right| \leq \left\{ \begin{array}{ll}
			C_{\alpha,p} \, \| g \|_{h^p} &\text{if} \;\; |\alpha|<\gamma_p; \\
			C_{\alpha,p} \, \| g \|_{h^p} \, \displaystyle \bigg[\log \left(1+\frac{1}{r}\right)\bigg]^{-1/p} &\text{if} \;\; |\alpha|=\gamma_p=N_p.
		\end{array} \right.
	\end{equation}
\end{proposition}

	Note that condition \eqref{eq:molecule-ball} for $|\alpha|=\gamma_p=N_p$ gets stronger as $r\rightarrow 0$.
	
\begin{proof}
Since $g$ is a distribution of compact support and hence acts on $C^\infty(\Rn)$, we can define
$\langle g, \, (\cdot-x_0)^{\alpha} \rangle$ unambiguously for any multi-index $\alpha \in \Z_{+}^{n}$, and $\langle g, \, (\cdot-x_0)^{\alpha} \rangle = \langle g, \, \phi \rangle$
for all $\phi \in C^{\infty}(\Rn)$ such that $\phi(y) = (y-x_0)^{\alpha}$ on the support of $g$.

By a translation argument we may assume that $x_0=0$. For each unit vector on $v\in \mathbb{S}^{n-1}$ and $\alpha \in \Z_{+}^{n}$ such that $|\Ga|\leq N_p$, we choose $\phi_0^{v,\Ga}$ satisfying the following conditions:
\begin{itemize}
	\item[\textnormal{(i)}] $\phi_0^{v,\Ga} \in C^{\infty}_{c}(\Rn)$ with support in $B\left(\frac{v}{2}, 2\right)$ and $\|\partial^{\beta} \, \phi_0^{v,\Ga}\|_{L^{\infty}}\leq 
	2^{|\beta|-2n}$ for all $|\beta|\leq N_p+1$;
	\item[\textnormal{(ii)}] $\phi_0^{v,\Ga}(y)= C_{\alpha} \, y^{\Ga}$ for all $|y|<1$ for some constant $C_\alpha$ depending only on $n$ and $\alpha$;
	\item[\textnormal{(iii)}] $\int \phi_0^{v,\Ga}(y)dy\neq 0$.
\end{itemize}
For each $x \in \Rn$ with $|x| > \frac{r}{2}$ we define
$$
\phi^{x,\Ga}(y)=\frac{1}{|x|^{n}} \,\, \phi^{\frac{x}{|x|},\Ga}_0\bigg(\frac{y}{2|x|}\bigg).
$$
We claim  $\phi^{x,\Ga} \in \mathcal{F}_{k}^{T,\,x}$ for $T=2$ and $k\leq N_{p}+1$. Indeed, note first that $\supp(\phi^{x,\Ga}) \subset B(x,t)$ for $t=4|x|$  since if $|y-x|>t$ we have
$$
\left| \frac{y}{2|x|} - \frac{x}{2|x|} \right| = \frac{|y-x|}{2|x|} \geq \frac{t}{2|x|} = 2
$$ 
and then $\phi^{\frac{x}{|x|},\Ga}_0({y}/{2|x|})=0$. Moreover, for $|\beta| \leq N_p+1$, by assumption (i),
$$
\left\| \partial^{\beta} \, \phi^{x,\Ga} \right\|_{L^{\infty}} = 2^{-|\beta|}|x|^{-n-|\beta|} \left\| \partial^{\beta} \, \phi_{0}^{\frac{x}{|x|},\Ga} \right\|_{L^{\infty}} \leq t^{-n-|\beta|}.
$$

On the support of $g$, $|y|<r$ and $|x|>\frac{r}{2}$ so $\frac{|y|}{2|x|} < 1$ and  by assumption (ii), $\phi^{x,\Ga}(y)=\frac{C_{\alpha} \, y^{\Ga}}{|x|^{n+|\alpha|}}$. Hence 
\begin{align*}
	m_{\mathcal{F}_k}(g)(x)=\sup_{\phi \, \in \, \mathcal{F}^{T,x}_{k}} \left| \langle g,\,\phi \rangle \right| &\geq \left| \langle g,\,\phi^{x,\Ga} \rangle \right| = C_\alpha|x|^{-n-|\alpha|} \left| \langle g, \, (\cdot - x_0)^\alpha \rangle \right|.
\end{align*}

When $|\alpha|=\gamma_p=N_p$, this gives
\begin{align*}
	\| g \|_{h^p}^{p} &\geq \int_{\frac{r}{2}<|x|<\frac{r+1}{2}} [m_{\mathcal{F}_{k}}(f)(x)]^pdx \\
& \geq  C_\alpha \left| \langle g, \, (\cdot - x_0)^\alpha \rangle \right|^{p} \, \int_{\frac{r}{2}<|x|\leq \frac{r+1}{2}} |x|^{-p(n+|\alpha|)}dx\\
& \geq C_\alpha \left| \langle g, \, (\cdot - x_0)^\alpha \rangle \right|^{p} \, \log \left(1+ \frac{1}{r} \right).
\end{align*}

When $|\alpha|<\gamma_p$, 
we consider $1<|x|<\frac{3}{2}$. Since in particular $|x|>\frac{r}{2}$, the same calculations as above give 
\begin{align*}
	\| g \|_{h^p}^{p} \geq \int_{1<|x|<\frac{3}{2}} [m_{\mathcal{F}_{k}}(g)(x)]^pdx \geq   C_\alpha \left| \langle g, \, (\cdot - x_0)^\alpha \rangle \right|^{p} \, \int_{1<|x|\leq \frac{3}{2}} |x|^{-p(n+|\alpha|)}dx =  C_{n,\alpha,p} \left| \langle g, \, (\cdot - x_0)^\alpha \rangle \right|^{p}.
\end{align*}
\end{proof}

We will now show that the result above can be extended to a class of $h^p$ distributions which we will call {\em pseudo-molecules} and which do not necessarily have compact support.

\begin{definition} 
\label{pseudo-molecule}
	Fix some constant $\fC > 0$. We say that $\fM \in \mathcal{S}'(\Rn)$ is a pseudo-molecule in $h^p(\Rn)$ associated to the ball $B \subset \Rn$ if $ \,\fM=g + h$ in $\mathcal{S}'(\Rn)$, where $g \in h^p(\Rn)$ has support in $B$, $h\in H^p(\Rn)$, and 	
	$$
	\|g\|_{h^p} + \|h\|_{H^p} \leq \fC.
	$$
\end{definition}

\begin{proposition} 
	\label{proposition:pseudo-molecule-moment-estimate}
	Let $0<p\leq 1$ and $\fM$ a pseudo-molecule in $h^p(\Rn)$ associated to the ball $B=B(x_0,r)$ with $0<r<1$. Then for $\alpha \in \Z_+^n$, $|\alpha| \leq N_p$, the moments $\langle \fM, \, (\cdot-x_0)^{\alpha} \rangle$ are well-defined and satisfy
	\begin{align} 
	\label{estimate-molecule}
			\left|\left\langle \fM, \, (\cdot-x_0)^{\alpha} \right\rangle \right| \lesssim \left\{ \begin{array}{ll}
				C_{\alpha,p} \, \fC &\text{if} \;\; |\alpha|<\gamma_p; \\
				\textcolor{white}{.} & \textcolor{white}{.} \\
				\displaystyle C_{\alpha,p} \, \fC \left[\log \left(1+\frac{1}{r}\right)\right]^{-1/p} &\text{if} \;\; |\alpha|=\gamma_p=N_p.
			\end{array} \right.
	\end{align}
\end{proposition}

\begin{proof}
Writing $\fM= g+ h$ as in Definition~\ref{pseudo-molecule}, since $h\in H^p(\Rn)$ satisfies vanishing moment conditions up the order $N_p$, we have 
$\langle h, \, (\cdot-x_0)^{\alpha} \rangle = 0$ (the pairing here is the one between $H^p$ and its dual space, the homogeneous Lipschitz space $\dot{\Lambda}_{n(1/p-1)}$).

For $g \in h^p(\Rn)$ supported in $B$, the moments $\langle g, \, (\cdot-x_0)^{\alpha} \rangle$ can be defined as in
Proposition \ref{lemma:moment-log-control}.  Thus we can set
$$\langle \fM, \, (\cdot-x_0)^{\alpha} \rangle: = \langle g, \, (\cdot-x_0)^{\alpha} \rangle + \langle h, \, (\cdot-x_0)^{\alpha} \rangle = \langle g, \, (\cdot-x_0)^{\alpha} \rangle.$$
If $\fM$ has an alternative decomposition $g' + h'$ satisfying the conditions of  Definition~\ref{pseudo-molecule}, then we must have that $g - g' \in H^p$ and therefore the moments of $g'$ are the same as those of $g$.

The estimates \eqref{estimate-molecule} now follow immediately from \eqref{eq:molecule-ball}. 
\end{proof}

\begin{proof}[Proof of Theorem~\ref{theorem:converse}]
Let $0<p\leq 1$ and $T$ be a linear and bounded operator on $h^p(\R^n)$ that maps each $(p,2)$ atom in $h^p(\Rn)$ into a pseudo-molecule centered in the same ball as the support of the atom.  

As \cite[Definition 5.1 and Proposition 5.2]{DLPV} rely on the specific form of the operators considered there, namely those with a nice kernel, we first need to make sense of the cancellation conditions \eqref{inhomogeneous-condition} in this more general context.  Fix $\alpha \in \Z_{+}^{n}$ with $|\alpha|\leq N_p$.  We want to show
$T^{\ast}\left[(\cdot-x_0)^{\alpha}\right]$ is well defined locally in the following sense.  

Fix a ball $B=B(x_0,r) \subset \R^n$ with $r<1$.  We will show that $T^{\ast}\left[(\cdot-x_0)^{\alpha}\right]$ can be identified with $f$ in $(L^{2}_{N_p}(B))^*$.  Here $L^{2}_{N_p}(B)$ denotes the space of functions in $L^2(B)$ with vanishing moments up to order $N_p$,  and its dual space  can be identified with the quotient of $L^2(B)$ by the subspace ${\mathcal P}_{N_p}$ of polynomials of order up to $N_p$.  We then have
\begin{equation}
\label{eq:duality}
\|f\|_{(L^{2}_{N_p}(B))^*} : = \sup_{\substack{\psi \in L^{2}_{N_p}(B) \\ \| \psi \|_{L^2(B)}\leq 1}} \left| \left\langle f, \, \psi \right\rangle \right|  = \inf_{P \in {\mathcal P}_{N_p}}\|f - P\|_{L^2(B)} = \|f - P_{B}^{N_p}(f)\|_{L^2(B)}, 
\end{equation}
where $P_{B}^{N_p}(f)$ is the  element of ${\mathcal P}_{N_p}$ with the same moments as $f$ over $B$ up to order $N_p$.

Given a  $\psi \in L^{2}_{N_p}(B)$ with $\|\psi\|_{L^2(B)} \leq 1$, let
$$
a(x) = \psi(x) \, |B|^{\frac{1}{2}-\frac{1}{p}}.
$$
Note that $a$ is a $(p,2)$ atom supported on $B$  (strictly speaking we have $\supp(a) \subset \overline{B}$ but in the calculation of the norm we may always take $\psi$ of compact support in $B$).  By the boundedness assumptions on $T$,   $\|Ta\|_{h^p(\Rn)} \lesssim \|a\|_{h^p(\Rn)} \leq C$ independent of $a$ and $\fM = Ta$ is a pseudo-molecule, where the choice of the constant $\fC$ in Definition~\ref{pseudo-molecule} should be consistent with the norm of $T$.
Thus by \eqref{estimate-molecule},
\begin{align*}
\left| \left\langle T^{\ast}[(\cdot-x_0)^{\alpha}], \, a \right\rangle \right| 
&: = \left|  \left\langle (\cdot-x_0)^{\alpha}, \, Ta \right\rangle \right|  \\
&\leq \left\{ \begin{array}{ll} 
		C_{\alpha,p} \,\fC &\quad \text{if } |\alpha|< \gamma_p, \\
		C_{\alpha,p} \, \fC \bigg[ \log\bigg(1+\frac{1}{r}\bigg)\bigg]^{-1/p} &\quad \text{if } |\alpha| = \gamma_p = N_p.
	\end{array} \right. 
\end{align*}
Replacing $a$ by $\psi$, we see that the left-hand-side defines a bounded linear functional $f \in (L^{2}_{N_p}(B))^*$ with
\begin{align*}
\left| \left\langle f, \, \psi \right\rangle \right| 
&= |B|^{\frac{1}{p}-\frac{1}{2}}\left| \left\langle T^{\ast}[(\cdot-x_0)^{\alpha}], \, a \right\rangle \right| 
&\leq \left\{ \begin{array}{ll} 
		C_{\alpha,p} |B|^{\frac{1}{p}-\frac{1}{2}}\,\fC &\quad \text{if } |\alpha|< \gamma_p, \\
		C_{\alpha,p} |B|^{\frac{1}{p}-\frac{1}{2}}\, \fC \bigg[ \log\bigg(1+\frac{1}{r}\bigg)\bigg]^{-1/p} &\quad \text{if } |\alpha| = \gamma_p = N_p.
	\end{array} \right. 
\end{align*}
Thus by \eqref{eq:duality}, we have
\begin{align*}
	\left( \fint_{B} |f-P_{N_p}(f)|^{2} \right)^{1/2} &= |B|^{-\frac{1}{2}}\, \left( \int_{B} |f-P_{N_p}(f)|^{2} \right)^{1/2}  \\
	&= |B|^{-\frac{1}{2}}\, \sup_{\substack{\psi \in L^{2}_{N_p}(B) \\ \| \psi \|_{L^2(B)}\leq 1}} \left| \left\langle f, \, \psi \right\rangle \right|\\
		&\leq \left\{ \begin{array}{ll} 
		C_{n,p} \, r^{\gamma_p} &\quad \text{if } |\alpha|< \gamma_p, \\
		C_{n,p} \, r^{\gamma_p} \bigg[ \log\bigg(1+\frac{1}{r}\bigg)\bigg]^{-1/p} &\quad \text{if } |\alpha| = \gamma_p = N_p.
	\end{array} \right. \\
&= C_{n,p} \; \Psi_{p,\alpha}(r).
\end{align*}
\end{proof}

While the pseudo-molecules described above are not restricted by any size or decay conditions, the name is motivated by our main example, the $h^p$ molecules defined   
 in \cite[Definition 3.5]{DLPV}.  As will be seen below, these
are pseudo-molecules, even when no cancellation is assumed. 

\begin{definition} 
\label{definition:pre-molecule}
	Let $0<p\leq 1 \leq s< \infty$ with $p\neq s$, $\lambda > n \left({s}/{p}-1\right)$, and $C > 0$. We say that a measurable function $M$ is a $(p,s,\lambda, C)$ pre-molecule in $h^p(\Rn)$ if there exist a ball $B(x_0,r) \subset \Rn$ and a constant $C>0$ such that
	$$
	\textnormal{M1.} \;\; \|M\|_{L^s(B)} \leq C \, r^{n\left(\frac{1}{s}-\frac{1}{p}\right)} \quad \textnormal{M2.} \;\; \| M \, |\cdot - x_0|^{\frac{\lambda}{s}} \,  \|_{L^{s}(B^c)} \leq C \, r^{\frac{\lambda}s + n \left(\frac{1}{s}-\frac{1}{p}\right)}.
	$$
\end{definition}

\begin{lemma} 
\label{example:premolecule}
Let $M$ be a $(p,s,\lambda,C)$ pre-molecule centered in $B(x_0,r) \subset \Rn$.  Then $M$ is a pseudo-molecule, with the constant $\fC$ in Definition~\ref{pseudo-molecule} depending on $\|M\|_{h^p}$ and $C$. 
\end{lemma}

\begin{proof}
We may assume without loss of generality that $C  = 1$, allowing us to apply \cite[Proposition 3.7]{DLPV} to get the decomposition (in the sense of distributions)
		$$
		M = \sum_{j=1}^{\infty} c_j \, a_j \, + \, a_B,
		$$
		where each $a_j$ is a $(p,2)$ atom for $H^p(\Rn)$ (i.e.\ with full cancellation) supported in $B(x_0,2^{j}r)$, $\sum |c_j|^p \leq C_{n,p,\lambda}$, and 
		$a_B\in L^2(B)$.  
		
		By the atomic decomposition of $H^p$, we have $h = \sum_{j=1}^{\infty} c_j \, a_j \, \in H^p$ with $ \| h \|_{H^p}\lesssim C_{n,p,\lambda}$.  Moreover $g = a_B \in h^p$ and by the triangle inequality
$$
			\|g\|_{h^p} \leq \| M \|_{h^p} + \| h \|_{h^p} \lesssim \| M \|_{h^p} + C_{n,p,\lambda}.
$$
Thus $\fM = M$ satisfies the conditions of Definition~\ref{pseudo-molecule} with $\fC  \lesssim \| M \|_{h^p} + C_{n,p,\lambda}$.

It is important to note that since we are not assuming any cancellation conditions on $M$, hence none on $a_B$, we cannot conclude, as in \cite[Proposition 3.7]{DLPV} , that $\| M \|_{h^p} \lesssim 1$.		
\end{proof}

\begin{proof}[Proof of Corollary~\ref{corollary:converse-ICZOp}]
One direction follows from Lemma~\ref{example:premolecule} and Theorem~\ref{theorem:converse}: if $T:\mathcal{S}'(\Rn) \rightarrow \mathcal{S}'(\Rn)$  is bounded on $h^p(\Rn)$ and takes each $(p,2)$ atom to a pre-molecule centered in the same ball as the support of the atom, then it satisfies the hypotheses of the Theorem and the cancellation conditions \eqref{inhomogeneous-condition} hold.

For the converse we have to use results from \cite{DLPV}.  Suppose $T:\mathcal{S}'(\Rn) \rightarrow \mathcal{S}'(\Rn)$ is continuous.  If,  for some appropriate fixed constants $s,\lambda$ and $C$, $T$ takes each $(p,2)$ atom in $h^p(\Rn)$ to a $(p,s,\lambda,C)$ pre-molecule centered in the same ball as the support of the atom, and in addition it satisfies the cancellation conditions \eqref{inhomogeneous-condition}, then we want to show that it maps each $(p,2)$ atoms to a \textit{bona fide} molecule $M$ as in \cite[Definition 3.5]{DLPV}.  By \cite[Proposition 3.7]{DLPV} such a molecule will have $h^p$ norm bounded by a constant (depending on $s,\lambda$ and $C$), so the atomic decomposition and the continuity of $T$ on $\mathcal{S}'(\Rn)$ will give us the boundedness of $T$ on $h^p(\Rn)$.

Since the size conditions (M1) and (M2) in Definition~\ref{definition:pre-molecule} are identical to the ones in \cite[Definition 3.5]{DLPV}, it just remains to verify that the cancellation condition in the latter definition, (M3), holds for some $\omega$.  This follows from the cancellation conditions \eqref{inhomogeneous-condition} on $T$ in the same way as at the end of the proof of \cite[Theorem 5.3]{DLPV}.  That argument does not use the specific properties of $T$ besides the cancellation conditions and, of course, the definition of $T^{\ast}\left[(\cdot-x_0)^{\alpha}\right]$, which, as shown in the proof Theorem~\ref{theorem:converse} above,  is well defined precisely because $T$ takes atoms to pre-molecules, which are pseudo-molecules.  The constant $\omega$ in (M3) will depend on the constant $C$ in  \eqref{inhomogeneous-condition}.
\end{proof}

\begin{proof}[Proof of Theorem~\ref{theorem:thm2}]
 Assume $T: \mathcal{S}(\Rn) \rightarrow \mathcal{S}'(\Rn) $ is  a  strongly singular inhomogeneous Calder\'on--Zygmund operator.  This means it extends continuously from $L^2(\Rn)$ to itself, from $L^q(\Rn)$ to $L^2(\Rn)$, where
$$
\frac{1}{q} = \frac{1}{2} + \frac{\beta}{n}, \;\; \text{for some} \;\; \frac{n}{2}(1-\sigma) \leq \beta < \frac{n}{2}, \;\; 0<\sigma\leq 1,
$$
its distributional kernel $K$ agrees with a continuous function away of the diagonal in $\R^n \times \R^n$, and there exist $\mu>0$ and $0<\delta\leq 1$ such that 
\begin{equation} 
	\label{inhomogeneous-kernel}
	|K(x,y)| \leq C \ \min{\left\{ \frac{1}{|x-y|^n}, \ \frac{1}{|x-y|^{n+\mu}} \right\}}, \;\; \text{for} \;\; x \neq y,
\end{equation} 
and
\begin{equation} \label{holder-condition}
	|K(x,y)-K(x,z)|+|K(y,x)-K(z,x)| \leq C \dfrac{|y-z|^{\delta}}{|x-z|^{n+ \frac{\delta}{\sigma}}}
\end{equation}
whenever $|x-z| \geq 2|y-z|^{\sigma}$.
The size conditions \eqref{inhomogeneous-kernel} and \eqref{holder-condition} on the kernel, together with the boundedness assumptions on $T$, with no further cancellation assumption, imply that if $a$ is a $(p,2)$ atom in $h^p(\Rn)$, then $Ta$ satisfies the size conditions of a molecule in $h^p(\Rn)$, namely (M1) and (M2) in Definition~\ref{definition:pre-molecule}, as shown in the proofs of \cite[Theorem 5.3 and 5.8]{DLPV}).  The desired result is then a consequence of Corollary~\ref{corollary:converse-ICZOp}.
\end{proof}

\end{document}